\documentclass[12pt]{amsart}
\usepackage{cases}
\usepackage{cases}
\usepackage{cases}
\usepackage{mathrsfs}
\usepackage{amsfonts}
\usepackage{amscd}
\usepackage{tikz}

\usepackage{latexsym,amssymb}
\usepackage{a4wide}
\usepackage{epic,eepic,latexsym, amssymb, amscd, amsfonts, xypic, floatflt}
\usepackage{amsthm}
\usepackage{amsmath}
\usepackage{amscd}
\usepackage{graphicx}
\usepackage[mathscr]{eucal}
\input xy
\xyoption{all}
\usepackage{verbatim}

\numberwithin{equation}{section}
\begin{document}
\theoremstyle{plain}
\newtheorem{thm}{Theorem}[section]
\newtheorem{theorem}[thm]{Theorem}
\newtheorem{lemma}[thm]{Lemma}
\newtheorem{corollary}[thm]{Corollary}
\newtheorem{proposition}[thm]{Proposition}
\newtheorem{conjecture}[thm]{Conjecture}
\theoremstyle{definition}
\newtheorem{remark}[thm]{Remark}
\newtheorem{remarks}[thm]{Remarks}
\newtheorem{definition}[thm]{Definition}
\newtheorem{example}[thm]{Example}

\newcommand{\caA}{{\mathcal A}}
\newcommand{\caB}{{\mathcal B}}
\newcommand{\caC}{{\mathcal C}}
\newcommand{\caD}{{\mathcal D}}
\newcommand{\caE}{{\mathcal E}}
\newcommand{\caF}{{\mathcal F}}
\newcommand{\caG}{{\mathcal G}}
\newcommand{\caH}{{\mathcal H}}
\newcommand{\caI}{{\mathcal I}}
\newcommand{\caJ}{{\mathcal J}}
\newcommand{\caK}{{\mathcal K}}
\newcommand{\caL}{{\mathcal L}}
\newcommand{\caM}{{\mathcal M}}
\newcommand{\caN}{{\mathcal N}}
\newcommand{\caO}{{\mathcal O}}
\newcommand{\caP}{{\mathcal P}}
\newcommand{\caQ}{{\mathcal Q}}
\newcommand{\caR}{{\mathcal R}}
\newcommand{\caS}{{\mathcal S}}
\newcommand{\caT}{{\mathcal T}}
\newcommand{\caU}{{\mathcal U}}
\newcommand{\caV}{{\mathcal V}}
\newcommand{\caW}{{\mathcal W}}
\newcommand{\caX}{{\mathcal X}}
\newcommand{\caY}{{\mathcal Y}}
\newcommand{\caZ}{{\mathcal Z}}
\newcommand{\fA}{{\mathfrak A}}
\newcommand{\fB}{{\mathfrak B}}
\newcommand{\fC}{{\mathfrak C}}
\newcommand{\fD}{{\mathfrak D}}
\newcommand{\fE}{{\mathfrak E}}
\newcommand{\fF}{{\mathfrak F}}
\newcommand{\fG}{{\mathfrak G}}
\newcommand{\fH}{{\mathfrak H}}
\newcommand{\fI}{{\mathfrak I}}
\newcommand{\fJ}{{\mathfrak J}}
\newcommand{\fK}{{\mathfrak K}}
\newcommand{\fL}{{\mathfrak L}}
\newcommand{\fM}{{\mathfrak M}}
\newcommand{\fN}{{\mathfrak N}}
\newcommand{\fO}{{\mathfrak O}}
\newcommand{\fP}{{\mathfrak P}}
\newcommand{\fQ}{{\mathfrak Q}}
\newcommand{\fR}{{\mathfrak R}}
\newcommand{\fS}{{\mathfrak S}}
\newcommand{\fT}{{\mathfrak T}}
\newcommand{\fU}{{\mathfrak U}}
\newcommand{\fV}{{\mathfrak V}}
\newcommand{\fW}{{\mathfrak W}}
\newcommand{\fX}{{\mathfrak X}}
\newcommand{\fY}{{\mathfrak Y}}
\newcommand{\fZ}{{\mathfrak Z}}

\newcommand{\bA}{{\mathbb A}}
\newcommand{\bB}{{\mathbb B}}
\newcommand{\bC}{{\mathbb C}}
\newcommand{\bD}{{\mathbb D}}
\newcommand{\bE}{{\mathbb E}}
\newcommand{\bF}{{\mathbb F}}
\newcommand{\bG}{{\mathbb G}}
\newcommand{\bH}{{\mathbb H}}
\newcommand{\bI}{{\mathbb I}}
\newcommand{\bJ}{{\mathbb J}}
\newcommand{\bK}{{\mathbb K}}
\newcommand{\bL}{{\mathbb L}}
\newcommand{\bM}{{\mathbb M}}
\newcommand{\bN}{{\mathbb N}}
\newcommand{\bO}{{\mathbb O}}
\newcommand{\bP}{{\mathbb P}}
\newcommand{\bQ}{{\mathbb Q}}
\newcommand{\bR}{{\mathbb R}}
\newcommand{\bT}{{\mathbb T}}
\newcommand{\bU}{{\mathbb U}}
\newcommand{\bV}{{\mathbb V}}
\newcommand{\bW}{{\mathbb W}}
\newcommand{\bX}{{\mathbb X}}
\newcommand{\bY}{{\mathbb Y}}
\newcommand{\bZ}{{\mathbb Z}}
\newcommand{\id}{{\rm id}}

\title[Volume conjecture for $SU(n)$-invariants] {Volume conjecture for $SU(n)$-invariants}

\author[Qingtao Chen, Kefeng Liu,
Shengmao Zhu]{Qingtao Chen, Kefeng Liu, Shengmao Zhu}

\address{Department of Mathematics \\
ETH Zurich \\
8092 Zurich \\ Switzerland } \email{qingtao.chen@math.ethz.ch}

\address{Center of Mathematical Sciences \\
Zhejiang University, Box 310027 \\
Hangzhou, P. R. China. }
\address{Department of mathematics \\
University of California at Los Angeles, Box 951555\\
Los Angeles, CA, 90095-1555.} \email{liu@math.ucla.edu}

\address{Center of Mathematical Sciences \\
Zhejiang University, Box 310027 \\
Hangzhou, P. R. China. } \email{zhushengmao@gmail.com}

\begin{abstract}
This paper discuss an intrinsic relation among congruent relations \cite{CLPZ}, cyclotomic expansion and Volume Conjecture for $SU(n)$ invariants. Motivated by the congruent relations for $SU(n)$ invariants obtained
in our previous work \cite{CLPZ}, we study certain limits of the
$SU(n)$ invariants at various roots of unit. First, we prove a new
symmetry property for the $SU(n)$ invariants by using a symmetry of
colored HOMFLYPT invariants. Then we propose some conjectural
formulas including the cyclotomic expansion conjecture and volume
conjecture for $SU(n)$ invariants (specialization of colored
HOMFLYPT invariants). We also give the proofs of these conjectural
formulas for the case of figure-eight knot.

\end{abstract}

\maketitle



\section{Introduction}
In our previous work joint with P. Peng \cite{CLPZ}, we introduced
the $SU(n)$ quantum invariant for a link $\mathcal{L}$ as follow:
\begin{align} \label{defSU(n)}
&J_{N}^{SU(n)}(\mathcal{L};q)=\left( \frac{q^{-2lk(\mathcal{L})\kappa _{(N)}}t^{-2lk(%
\mathcal{L})N}W_{(N)(N),...,(N)}(\mathcal{L};q,t)}{s_{(N)}(q,t)}\right)
|_{t=q^{n}}\\\nonumber
&=\frac{q^{-2lk(\mathcal{L})N(N-1)}q^{-2nlk(\mathcal{L})N}W_{(N)(N),...,(N)}(%
\mathcal{L};q,q^{n})}{s_{(N)}(q,q^{n})}\\\nonumber &
=\frac{q^{-2lk(\mathcal{L})(N(N-1)+nN)}W_{(N)(N),...,(N)}(\mathcal{L},q,q^{n})}{%
s_{(N)}(q,q^{n})}
\end{align}
where $W_{(N)(N),...,(N)}(\mathcal{L},q,t)$ is the colored HOMFLYPT
invariants of $\mathcal{L}$, see Section 2.1 for the definitions. In
particular, when $n=2$,
$J_N^{SU(2)}(\mathcal{L};q)=J_N(\mathcal{L};q)$ is the classical
(reduced) colored Jones polynomial with a suitable variable changes,
see Section 7 in \cite{CLPZ} for detail.  In this paper, by using
one of the symmetries of the colored HOMFLYPT invariants obtained in
\cite{CLPZ}, we prove the following  symmetry of the $SU(n)$
invariant about the rank $n$.
\begin{theorem}
For a knot $\mathcal{K}$ and two integers $n\geq m\geq 2$, we have
\begin{align}
J_{N}^{SU(n)}(\mathcal{K};q)\equiv J_{N}^{SU(n-m)}(\mathcal{K};q)
\mod [m].
\end{align}
\end{theorem}

\begin{remark}
As in \cite{CLPZ}, we use the notation $[m]=q^{m}-q^{-m}$ throughout
this paper, and $A\equiv B \mod [m]$ denotes $\frac{A-B}{C} \in
\mathbb{Z}[q,q^{-1}]$.
\end{remark}
Formula (1.2) can be viewed as a new congruent relation with respect to the
rank $n$. In \cite{CLPZ}, we have proposed the following congruent
relation for $J_{N}^{SU(n)}(\mathcal{K};q)$ which reveals some
symmetries with respect to the color $N$:
\begin{align}
J_{N}^{SU(n)}(\mathcal{K};q)-J_k^{SU(n)}(\mathcal{K};q)\equiv 0 \mod
[N-k][N+k+n],
\end{align}
where $N\geq k\geq 0$. In fact, the congruent relation (1.3) is an
easy consequence of the following more general conjecture.
\begin{conjecture}[cyclotomic expansions for
$SU(n)$ invariants] For any knot $\mathcal{K}$, there exist
$H_k^{(n)}(\mathcal{K})\in \mathbb{Z}[q,q^{-1}]$, independent of $N$
($N\geq 0$). Such that
\begin{align}
J_{N}^{SU(n)}(\mathcal{K};q)=\sum_{k=0}^{N}C_{N+1,k}^{(n)}H_k^{(n)}(\mathcal{K}),
\end{align}
where $C_{N+1,k}^{(n)}=[N-(k-1)][N-(k-2)]\cdots
[N-1][N][N+n][N+n+1]\cdots [N+n+(k-1)]$, for $k=1,..,N$, and
$C_{N+1,0}^{(n)}=1$.  In particular,
$J_0^{SU(n)}(\mathcal{K};q)=H_0^{(n)}(\mathcal{K})=1$.
\end{conjecture}
Conjecture 1.3 is a generalization of the cyclotomic expansion for
colored Jones polynomials due to K. Habiro \cite{Hab}. By some
direct calculations, we find that Conjecture 1.3 holds for
figure-eight knot $4_1$ and trefoil knot $3_1$. See examples 2.5 and
2.6 in Section 2.

Next, we study the limit behaviors of  $SU(n)$ invariant
$J_{N}^{SU(n)}(\mathcal{K};q)$ at various roots of unit. For
convenience,  we introduce the notation
$\xi_{N,a}(s)=\exp(\frac{s\pi \sqrt{-1}}{N+a})$, where $a,s\in
\mathbb{Z}$. Then for a fixed $n\geq 2$, we have
\begin{conjecture}
(i) If $a\in \mathbb{Z}\backslash \{1,2,..,n-1\}$, then for any knot
$\mathcal{K}$:
\begin{align}
2\pi s\lim_{N\rightarrow \infty}\frac{\log
J_{N}^{SU(n)}(\mathcal{K};\xi_{N,a}(s))}{N+1}=0.
\end{align}
(ii) If $a\in \{1,2,...,n-1\}$, then
\begin{align}
2\pi s \lim_{N\rightarrow \infty} \frac{\log
J_{N}^{SU(n)}(\mathcal{K};\xi_{N,a}(s))
}{N+1}=Vol(S^3/\mathcal{K})+\sqrt{-1}CS(S^3/\mathcal{K})
\end{align}
for any hyperbolic knot $\mathcal{K}$.
\end{conjecture}

Conjecture 1.4 is a parallel generalization of the complex volume
conjecture for colored Jones polynomial\cite{Kashaev}\cite{MuMu}\cite{MMOTY}. In fact, when $n=2$,
part (i) of the Conjecture 1.4 is an easy consequence of the
congruent relation obtained in \cite{CLPZ}, see Section 2 for a
proof. Moreover,  for general $n\geq 2$, the congruent relation
(1.3) leads to  part (i) of the Conjecture 1.4 for $SU(n)$
invariant $J^{SU(n)}_{N}(\mathcal{K})$ similarly. As for part
(ii) of Conjecture 1.4, we believe it also holds for hyperbolic
links although we have only checked the case of knots. In the paper
\cite{Kaw2}, K. Kawagoe first proposed a special case of the part
(ii) of Conjecture 1.4 for $a=n-1$ and $s=1$.

In Section 3, we will prove that
\begin{theorem}
The Conjecture 1.4 holds for the figure-eight knot $4_1$.
\end{theorem}

We note that in the $SU(n)$ invariants $J_{N}^{SU(n)}$, there are two
variables $N,n$.  So it is natural to consider the double limits
$N,n\rightarrow \infty$.  In \cite{Kaw2}, K. Kawagoe has studied the
double limit of
\begin{align}
2\pi  \lim_{N\rightarrow \infty} \frac{\log
J_{N}^{SU(n)}(\mathcal{K};\xi_{N,n-1}(1)) }{N+1}
\end{align}
when $\frac{n-1}{N+n-1}$ keeps a different ratio.  In this paper, we
find certain double limit of the $SU(n)$ invariants
$J_{N}^{SU(n)}$ will also converge to zero or to the volume of the
hyperbolic knot complement. More precisely, we propose

\begin{conjecture}
Fix an integer $n\geq 2$,

(i) If $a\in \mathbb{Z}\backslash \{1,2,..,n-1\}$, then for any knot
$\mathcal{K}$:
\begin{align}
2\pi s\lim_{N\rightarrow \infty}\frac{\log
J_{N}^{SU(N+a+n)}(\mathcal{K};\xi_{N,a}(s))}{N+1}=0.
\end{align}
(ii) If $a\in \{1,2,...,n-1\}$, then
\begin{align}
2\pi s \lim_{N\rightarrow \infty} \frac{\log
J_{N}^{SU(N+a+n)}(\mathcal{K};\xi_{N,a}(s))
}{N+1}=Vol(S^3/\mathcal{K})+\sqrt{-1}CS(S^3/\mathcal{K})
\end{align}
for any hyperbolic knot $\mathcal{K}$.
\end{conjecture}

As an application of Theorem 1.1, it is direct to prove
\begin{theorem}
Conjecture 1.6 is equivalent to Conjecture 1.4.
\end{theorem}

{\bf Acknowledgements.}  We would like to thank K. Kawagoe for
sending his paper \cite{Kaw1} to our attention. We would like to thank Nicolai Reshetikhin, Rinat Kashaev, Giovanni Felder, Jun Murakami, and Tian Yang for their interests and helpful discussion. The research of K. Liu is supported by an NSF grant. The research of S. Zhu is supported by the National Science Foundation of China grant
No. 11201417 and the China Postdoctoral Science special Foundation No. 2013T60583.

\section{Symmetry, Congruent relations and cyclotomic expansions}
\subsection{Definitions}
In this subsection, we first review the definition of the colored
HOMFLYPT invariant of a  link $\mathcal{L}$ with $L$ components
$\mathcal{K}_1,...,\mathcal{K}_L$ (See Section 2 of \cite{CLPZ} for
detail).

Let $(\mathcal{L}_+,\mathcal{L}_-,\mathcal{L}_0)$ be a Conway
triple. The (unreduced framed) HOMFLYPT polynomial
$\mathcal{H}(\mathcal{L};q,t)$ of $\mathcal{L}$ is determined by the
following skein relation:
\begin{align*}
t\mathcal{H}(\mathcal{L}_+;q,t)-t^{-1}\mathcal{H}(\mathcal{L}_-;q,t)=(q-q^{-1})\mathcal{H}(\mathcal{L}_0;q,t)
\end{align*}
with $\mathcal{H}(U;q,t)=\frac{t-t^{-1}}{q-q^{-1}}$, where $U$
denotes an unknot.

The colored HOMFLYPT invariant can be defined by satellite knot. A
satellite of a knot $\mathcal{K}$ is  $\mathcal{K}\star Q$ which
means the knot $\mathcal{K}$ decorated by a diagram $Q$ in the
annulus. (see Figure 1 for a framed trefoil $\mathcal{K}$ decorated
by skein element $Q$).
\begin{figure}[!htb]
\begin{align*}
\mathcal{K} \qquad\qquad\qquad\quad \mathcal{Q}
\qquad\qquad\qquad\quad
\mathcal{K}\star \mathcal{Q}\\
\includegraphics[width=50 pt]{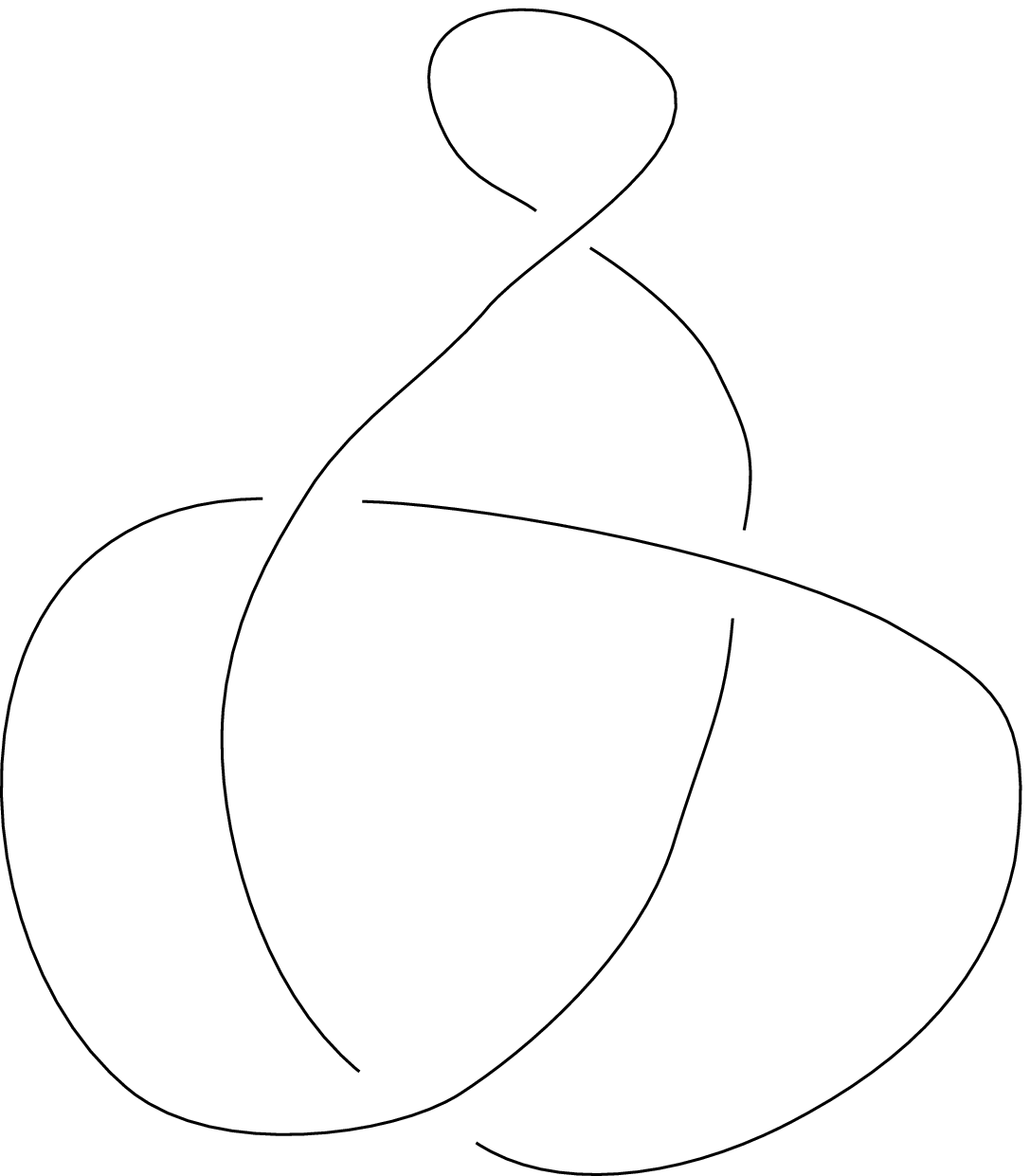}\qquad\qquad \includegraphics[width=50
pt]{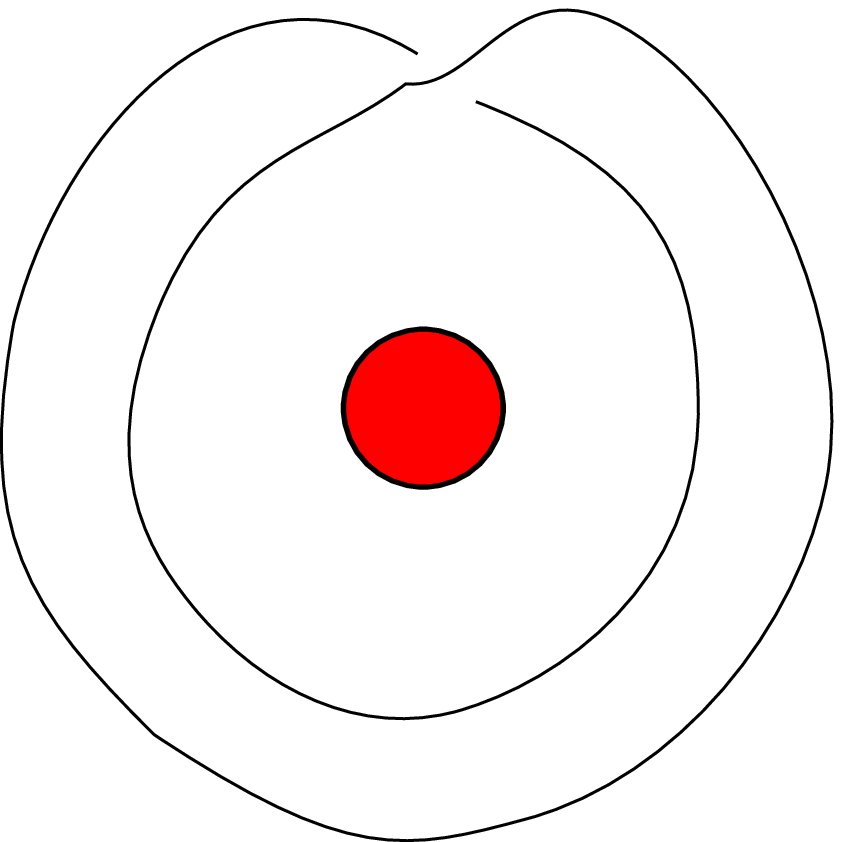} \qquad\quad \includegraphics[width=50
pt]{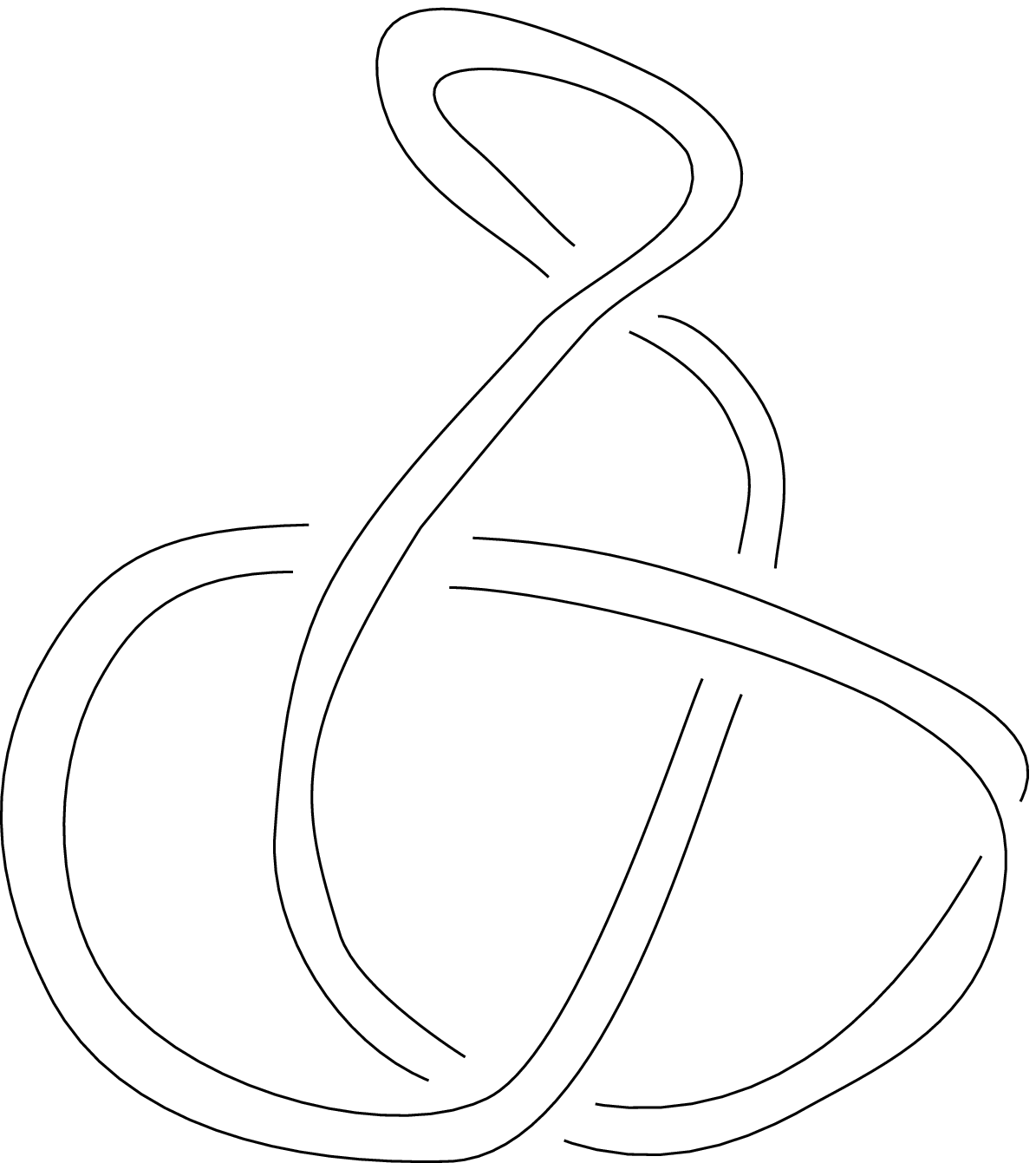}
\end{align*}
\caption{}
\end{figure}

There is a well-known set of idempotents, $E_{\lambda}$, for a
partition $\lambda$  of a positive integer $n$. The closure of
$E_{\lambda}$, denoted by $Q_{\lambda}$ forms a basis of the skein
of annulus $\mathcal{C}$ \cite{MA}. Given $L$ partitions
$\lambda^1,...,\lambda^L$, let
$\vec{\lambda}=(\lambda^1,...,\lambda^L)$, the colored HOMFLYPT
invariant of $\mathcal{L}$ (with color $\vec{\lambda}$) is defined
as
\begin{align*}
W_{\vec{\lambda}}(\mathcal{L};q,t)=q^{-\sum_{\alpha=1}^{L}\kappa_{\lambda^%
\alpha}w(\mathcal{K}_\alpha)} t^{-\sum_{\alpha=1}^L|\lambda^\alpha|w(%
\mathcal{K}_\alpha)} \mathcal{H}(\mathcal{L}\star
\otimes_{\alpha=1}^LQ_{\lambda^\alpha};q,t).
\end{align*}
where $w(\mathcal{K}_{\alpha})$ is the writhe number of knot
$\mathcal{K}_{\lambda^\alpha}$, and
$\kappa_{\alpha}=\sum_{i}\lambda^{\alpha}_i(\lambda^{\alpha}_i-2i+1)$,
$|\lambda^{\alpha}|=\sum_i\lambda^{\alpha}_{i}.$

As to the unknot $U$, its colored HOMFLYPT invariant( with color
$\lambda$) is given by
\begin{align*}
W_{\lambda}(U;q,t)=S_{\lambda}(q,t)=\sum_{\mu}\frac{\chi_{\lambda}(\mu)}{z_\mu}\prod_{j=1}
\frac{t^{\mu_j}-t^{-\mu_j}}{q^{\mu_j}-q^{-\mu_j}}.
\end{align*}
where $z_{\mu}=\prod_{j}\mu_j|Aut(\mu)|$ and $\chi_{\lambda}(\mu)$
is the character symmetric group.

\subsection{Symmetry}
In \cite{CLPZ}, we proved several symmetries for colored HOMFLYPT
invariants (see Theorem 1.2 in \cite{CLPZ}), and one of them is
\begin{align}
W_{\vec{\lambda}}(\mathcal{L};q,-t)&=(-1)^{\|\vec{\lambda}\|}W_{\vec{\lambda}}(\mathcal{L};q,t).
\end{align}
for a given link $\mathcal{L}$ with $L$ components, and
$\vec{\lambda}=(\lambda^1,..,\lambda^L)\in \mathcal{P}^L$,

As to a knot $\mathcal{K}$, we define the reduced colored HOMFLYPT
invariant with symmetric representation corresponding to a partition
$(N)$ as follow
\begin{align}
P_{N}(\mathcal{K};q,t)=\frac{W_{(N)}(\mathcal{K};q,t)}{S_{(N)}(q,t)}.
\end{align}
So by (2.1), we have
\begin{align}
P_{N}(\mathcal{K};q,t)=P_{N}(\mathcal{K};q,-t).
\end{align}
Then, by the definition (1.1)
\begin{align}
J_{N}^{SU(n)}(\mathcal{K};q)=P_{N}(\mathcal{K};q,q^n)=P_{N}(\mathcal{K};q,-q^n).
\end{align}

Now we let $q=e^{\frac{s\pi \sqrt{-1}}{m}}$, where $s\in
\mathbb{Z}$.
 If
$s$ odd,  then
\begin{align}
J_{N}^{SU(n)}(\mathcal{K};e^{\frac{s\pi
\sqrt{-1}}{m}})&=P_{N}(\mathcal{K};e^{\frac{s\pi
\sqrt{-1}}{m}},e^{\frac{ns\pi \sqrt{-1}}{m}})\\\nonumber
&=P_{N}(\mathcal{K};e^{\frac{s\pi \sqrt{-1}}{m}},-e^{\frac{(n-m)s\pi
\sqrt{-1}}{m}})\\\nonumber &=P_{N}(\mathcal{K};e^{\frac{s\pi
\sqrt{-1}}{m}},e^{\frac{(n-m)s\pi \sqrt{-1}}{m}})\\\nonumber
&=J_{N}^{SU(n-m)}(\mathcal{K};e^{\frac{s\pi \sqrt{-1}}{m}}),
\end{align}
and if $s$ even, it is obvious
\begin{align}
J_{N}^{SU(n)}(\mathcal{K};e^{\frac{s\pi
\sqrt{-1}}{m}})=J_{N}^{SU(n-m)}(\mathcal{K};e^{\frac{s\pi
\sqrt{-1}}{m}}).
\end{align}

In conclusion, we obtain
\begin{theorem}[=Theorem 1.1]
For a knot $\mathcal{K}$, $n\geq m\geq 2$, we have
\begin{align}
J_{N}^{SU(n)}(\mathcal{K};q)\equiv J_{N}^{SU(n-m)}(\mathcal{K};q)
\mod [m].
\end{align}
\end{theorem}

\subsection{Congruent relations and cyclotomic expansions}
At first, we review some results obtained in our previous joint work
with P. Peng \cite{CLPZ}. As to the colored Jones polynomial, we
have proved the following congruent relations for
$J_N(\mathcal{L};q)$. See Theorem 7.3 in  \cite{CLPZ}:
\begin{theorem}
For any knot $\mathcal{K}$, and $N\geq k\geq 0$, we have
\begin{align}
J_{N}(\mathcal{K};q)-J_k(\mathcal{K};q)\equiv 0 \mod [N-k][N+k+2].
\end{align}
\end{theorem}
In fact, it is a consequence of the cyclotomic expansion for the
colored Jones polynomial due to K. Habiro \cite{Hab}. For $N\geq 1$,
we define
\begin{align}
C_{N,k}=\prod_{j=1}^{k}(q^{2N}+q^{-2N}-q^{2j}-q^{-2j}).
\end{align}
In particular, $C_{N+1,0}=1$, $C_{N+1,N+1}=0$, and
\begin{align*}
C_{N+1,1}&=[N+2][N], \\
C_{N+1,2}&=[N+3][N+2][N][N-1], \\
\cdots \\
C_{N+1,k}&=[N+k+1][N+k]\cdots [N+2][N][N-1]\cdots [N-k+1] \\
\cdots \\
C_{N+1,N}&=[2N+1][2N]\cdots[N+2][N][N-1]\cdots [1].
\end{align*}
In our notation, Habiro's cyclotomic expansion of colored Jones
polynomial states:
\begin{theorem}[K. Habiro \cite{Hab}]
For any knot $\mathcal{K}$, there exist $H_k(\mathcal{K})\in
\mathbb{Z}[q,q^{-1}]$, independent of $N$ ($N\geq 0$). Such that
\begin{align}
J_{N}(\mathcal{K};q)=\sum_{k=0}^{N}C_{N+1,k}H_k(\mathcal{K}).
\end{align}
In particular, $J_0(\mathcal{K};q)=H_0(\mathcal{K})=1$.
\end{theorem}
In the next section, we will show  that, by using Theorem 2.2, we
can prove  part (i) of Conjecture 1.4 for $n=2$ case.

Now we consider the general $SU(n)$ invariants
$J_N^{SU(n)}(\mathcal{K};q)$. After some numerical computations, it
is natural to propose the following conjecture which is the
generalization of cyclotomic expansion (2.10) for colored Jones
polynomial.

\begin{conjecture}[=Conjecture 1.3] For any knot $\mathcal{K}$, there exist
$H_k^{(n)}(\mathcal{K})\in \mathbb{Z}[q,q^{-1}]$, independent of $N$
($N\geq 0$). Such that
\begin{align}
J_{N}^{SU(n)}(\mathcal{K};q)=\sum_{k=0}^{N}C_{N+1,k}^{(n)}H_k^{(n)}(\mathcal{K}),
\end{align}
where $C_{N+1,k}^{(n)}=[N-(k-1)][N-(k-2)]\cdots
[N-1][N][N+n][N+n+1]\cdots [N+n+(k-1)]$, for $k=1,..,N$, and
$C_{N+1,0}^{(n)}=1$.  In particular,
$J_0^{SU(n)}(\mathcal{K};q)=H_0^{(n)}(\mathcal{K})=1$.
\end{conjecture}

\begin{example}
For the figure-eight knot $4_1$, by using  formula (4) in
\cite{IMMM}, we have
\begin{align}
H_k^{(n)}(4_1)=\frac{[n-2+k]!}{[k]![n-2]!}\in \mathbb{Z}[q,q^{-1}].
\end{align}
\end{example}

\begin{example}
For the trefoil knot $3_1$, by using formula (3.61) in
\cite{FGS} with $t=-1$, we have
\begin{align}
H_k^{(n)}(3_1)=(-1)^kq^{k(2n+k-1)}\frac{[n-2+k]!}{[k]![n-2]!}\in
\mathbb{Z}[q,q^{-1}].
\end{align}
\end{example}
See Appendix 4.2 for more examples of cyclotomic expansions of
$SU(n)$ invariants.

By some direct calculations, we have
\begin{align}
[N+n][N]-[k+n][k]&=[N-k][N+k+n], \\
[N+n+1][N-1]-[k+n+1][k-1]&=[N-k][N+k+n],  \notag \\
\cdots,  \notag \\
[N+k+n-1][N-k+1]-[2k+n-1][1]&=[N-k][N+k+n],  \notag
\end{align}
for $N>k\geq 0$.

Therefore, if we assume Conjecture 2.4 holds,  by similar
method in the proof of Theorem 7.3 in \cite{CLPZ}, we obtain, for
any knot $\mathcal{K}$ and $N\geq k\geq 0$,
\begin{align}
J_{N}^{SU(n)}(\mathcal{K};q)-J_k^{SU(n)}(\mathcal{K};q)\equiv 0 \mod
[N-k][N+k+n].
\end{align}
which is the congruent relation for $SU(n)$ invariants obtained in
\cite{CLPZ}. See Conjecture 7.10 in \cite{CLPZ}.

\begin{remark}
After the completion of this paper, we note that in a recent paper
by S. Nawata and A. Oblomkov \cite{NO}, they propose a conjectural
cyclotomic expansion formula for colored HOMFLYPT invariants with
symmetric representation (See Conjecture 2.3 in \cite{NO}). Let
$a=q^n$ in their conjecture, we obtain
\begin{conjecture}
For any knot $\mathcal{K}$, there exist $\tilde{H}_k(\mathcal{K})\in
\mathbb{Z}[q,q^{-1}]$, independent of $N$ with $N\geq 0$, such that
\begin{align}
J_{N}^{SU(n)}(\mathcal{K};q)=\tilde{H}_0(\mathcal{K})+\tilde{C}_{N+1,1}^{(n)}\tilde{H}_1(\mathcal{K})
+\tilde{C}_{N+1,2}^{(n)}\tilde{H}_2(\mathcal{K})+\cdots+\tilde{C}_{N+1,N}^{(n)}\tilde{H}_{N}(\mathcal{K}),
\end{align}
where
$\tilde{C}_{N+1,k}^{(n)}=\frac{1}{[k]!}\prod_{j=0}^{k-1}[N-j]\prod_{j=0}^{k-1}[N+n+j]$
for $k=1,..,N$. In particular,
$J_0^{SU(n)}(\mathcal{K};q)=\tilde{H}_0(\mathcal{K})=1$.
\end{conjecture}
It is clear that Conjecture 2.8 is a weak version of Conjecture 2.4.
\end{remark}

\section{Congruent relations and limits of $SU(n)$ invariants}
In this section, we will show that the congruent relations is
closely related to part (i) of Conjecture 1.4.

Before discussing the limit behaviors of the general $SU(n)$
invariants, we first review the classical volume conjecture
\cite{Kashaev} for colored Jones polynomials \cite{MuMu}. In our
notations, it can be formulated as
\begin{conjecture}[Complex volume conjecture]
Let $\mathcal{L}$ be a hyperbolic link in $S^3$, then
\begin{align}
2\pi\lim_{N\rightarrow
\infty}\frac{J_{N}(\mathcal{L};\exp(\frac{\pi\sqrt{-1}}{N+1}))}{N+1}=Vol(S^3/\mathcal{L})+CS(S^3/\mathcal{L}).
\end{align}
\end{conjecture}

Many people have made a lot of efforts to prove the Conjecture 3.1,
see \cite{Mu} for a nice review. In the study of this conjecture,
it is natural to think  why we should take the unit root
$q=\frac{\pi \sqrt{-1}}{N+1}$ in $J_{N}(\mathcal{L};q)$, here the
label $N$ means the $N+1$-dimensional irreducible representations of
$U_{q}(sl_2)$. Namely, a natural question is what will happen if we
take the other roots of unity in the above limit? In this section,
we answer this question partially by the following theorem.
\begin{theorem}
Given a knot  $\mathcal{K}$ in $S^3$, let $a$ and $s$ be two fixed
integer and $a\neq 1$. If we take $q=\xi_{N,a}(s)=\exp \frac{s\pi
\sqrt{-1}}{N+a}$, then we have
\begin{align}
2\pi s\lim_{N\rightarrow
\infty}\frac{J_{N}(\mathcal{K};\xi_{N,a}(s))}{N+1}=0
\end{align}
\end{theorem}
\begin{proof}
In fact, it is an easy consequence of Theorem 2.2. For
$a=2,3,4,...$, by Theorem 2.2, we have
\begin{align}
J_{N}(\mathcal{K};\exp \frac{s\pi
\sqrt{-1}}{N+a})=J_{a-2}(\mathcal{K};\exp \frac{s\pi
\sqrt{-1}}{N+a})
\end{align}

For a fixed $a$, $J_{a-2}(\mathcal{K},q)$ is a fixed polynomial of
$q$, it is clear that the value of $J_{a-2}(\mathcal{K};\exp
\frac{s\pi \sqrt{-1}}{N+a})$ is bounded, i.e. there exists a
constant $C$ independent of the $N$ such that
\begin{align}
|J_{a-2}(\mathcal{K};\exp \frac{s\pi \sqrt{-1}}{N+a})|< C.
\end{align}
So we have
\begin{align}
\lim_{N\rightarrow \infty}\frac{J_{N}(\mathcal{K};\exp \frac{s\pi
\sqrt{-1}}{N+a})}{N+1}=\lim_{N\rightarrow
\infty}\frac{J_{a-2}(\mathcal{K};\exp \frac{s\pi
\sqrt{-1}}{N+a})}{N+1}=0
\end{align}

For $a=0,-1,-2,-3,....$, then $-a=0,1,2,3,..$. Theorem 2.2 also
gives
\begin{align}
J_{N}(\mathcal{K};\exp \frac{s\pi
\sqrt{-1}}{N-(-a)})=J_{-a}(\mathcal{K};\exp \frac{s\pi
\sqrt{-1}}{N+a}).
\end{align}
Similarly, we obtain
\begin{align}
\lim_{N\rightarrow \infty}\frac{J_{N}(\mathcal{K};\exp \frac{s\pi
\sqrt{-1}}{N+a})}{N+1}=\lim_{N\rightarrow
\infty}\frac{J_{-a}(\mathcal{K};\exp \frac{s\pi
\sqrt{-1}}{N+a})}{N+1}=0
\end{align}
\end{proof}
So we finish the proof of part (i) of Conjecture 1.4 for $n=2$.

We note that when $a=1$, the above limit is the complex  volume by
Conjecture 3.1. So in conclusion, the limit is equal to complex
volume for $a=1$ or is equal to zero if $a$ takes the other
integers. Why only $a=1$ gives the nonzero limit? We observe that
the mod term in formula (2.8) is given by $[N][N+2]$ if $k=0$, and
$N+1$ is the middle integer between $N$ and $N+2$.

Now it is natural to consider the case of $SU(n)$ invariant
$J_{N}^{SU(n)}(\mathcal{K})$. We observe that the mod term in the
congruent relation formula (2.15) is now given by $[N][N+n]$ when
$k=0$. So the middle integers between $N$ and $N+n$ contain
$N+1,N+2,...,N+n-1$. Therefore, as the application of congruent
relation formula (2.15), similarly we have, if $a\in
\mathbb{Z}\backslash \{1,2,..,n-1\}$, then
\begin{align}
2\pi s\lim_{N\rightarrow
\infty}\frac{J_{N}^{SU(n)}(\mathcal{K};\xi_{N,a}(s))}{N+1}=0.
\end{align}
Furthermore, as for $a\in \{1,2,..,n-1\}$, we believe that the above limit
is also convergent to the complex volume. So we  propose the
following conjecture:
\begin{conjecture}[=Conjecture 1.4]
(i) If $a\in \mathbb{Z}\backslash \{1,2,..,n-1\}$, then for any knot
$\mathcal{K}$:
\begin{align}
2\pi s\lim_{N\rightarrow \infty}\frac{\log
J_{N}^{SU(n)}(\mathcal{K};\xi_{N,a}(s))}{N+1}=0.
\end{align}
(ii) If $a\in \{1,2,...,n-1\}$, then
\begin{align}
2\pi s \lim_{N\rightarrow \infty} \frac{\log
J_{N}^{SU(n)}(\mathcal{K};\xi_{N,a}(s))
}{N+1}=Vol(S^3/\mathcal{K})+\sqrt{-1}CS(S^3/\mathcal{K})
\end{align}
for any hyperbolic knot $\mathcal{K}$.
\end{conjecture}
In the following, we will prove

\begin{theorem}
[=Theorem 1.5]The Conjecture 3.3 holds for the figure-eight knot $4_{1}$.
\end{theorem}

The whole proof of the Theorem 3.4 is long, so we divide the Theorem 3.4 into
the following two propositions and lemmas.

\begin{lemma}
For $a=1$ or $n-1$, we have the following
\begin{align}
\lim_{N\rightarrow\infty} \frac{\log J_{N-1}^{SU(n)}(4_{1};\xi_{N-1,a}(s))
}{N}=\frac{2}{\pi s}\int_{0}^{\frac{5\pi}{6}}\log2\sin(x)dx.
\end{align}

\end{lemma}

\begin{proof}
By formula (4) in \cite{IMMM}, we have
\begin{align}
J_{N-1}^{SU(n)}(4_{1})=\sum_{j=0}^{N-1}\prod_{k=1}^{j}\frac{[n-2+k]}%
{[k]}[N-k][N+(n-2)+k].
\end{align}
For convenience, we introduce the following notations
\begin{align}
f^{SU(n)}(N,k) & =\frac{[n-2+k]}{[k]}[N-k][N+(n-2)+k],\\
g^{SU(n)}(N,j) & =\prod_{k=1}^{j}f^{SU(n)}(N,k).\nonumber
\end{align}
When $q=\xi_{N-1,a}(s)=\exp(\frac{\pi s\sqrt{-1}}{N+\tilde{a}})$, where
$\tilde{a}=a-1$. We have
\begin{align}
[l]=q^{l}-q^{-l}=2\sqrt{-1}\sin\frac{ls\pi}{N+a}.
\end{align}

It is easy to show that
\begin{align}
f^{SU(n)}(N,k)=\frac{\sin\frac{(n-2+k)s\pi}{N+\tilde{a}}}{\sin\frac{ks\pi
}{N+\tilde{a}}}4\left( \sin\frac{(k+\tilde{a})s\pi}{N+\tilde{a}}\right) \left(
\sin\frac{((n-2)-\tilde{a}+k)s\pi}{N+\tilde{a}}\right) .
\end{align}

We prove a special case at first. Substituting $\tilde{a}=0$ ( or $ a=1$) to
the above formula, we have
\begin{align}
f^{SU(n)}(N,k)=\left( 2\sin\frac{(n-2+k)s\pi}{N}\right) ^{2}.
\end{align}
One can show that the function $g^{SU(n)}(N,j)=\prod_{k=1}^{j}f^{SU(n)}(N,k)$
takes the maximum value at $j=\lfloor(\frac{p}{s}-\frac{1}{6s})N-(n-2)\rfloor$
(here we can take the large $N$, such that $\frac{N}{s}$ are coprime.) for
some $1\leq p \leq s$.

Therefore, we have
\begin{align}
g^{SU(n)}(N,\lfloor(\frac{p}{s}-\frac{1}{6s})N-(n-2)\rfloor)\leq &
J_{N-1}^{SU(n)}(4_{1})\\
& \leq N\cdot g^{SU(N)}(N,\lfloor(\frac{p}{s}-\frac{1}{6s})N-(n-2)\rfloor
).\nonumber
\end{align}
Hence
\begin{align}
\frac{\log(g^{SU(n)}(N,\lfloor(\frac{p}{s}-\frac{1}{6s})N-(n-2)\rfloor)}{N} &
\leq\frac{\log(J_{N-1}^{SU(n)}(4_{1}))}{N}\\
& \leq\frac{\log(N)}{N}+ \frac{\log(g^{SU(N)}(N,\lfloor(\frac{p}{s}-\frac
{1}{6s})N-(n-2)\rfloor))}{N}.\nonumber
\end{align}

Since $\lim_{N\rightarrow\infty} \frac{\log(N)}{N}=0$, we have
\begin{align}
\lim_{N\rightarrow\infty} \frac{\log(J_{N-1}^{SU(n)}(4_{1}))}{N} &
=\lim_{N\rightarrow\infty}\frac{\log(g^{SU(n)}(N,\lfloor(\frac{p}{s}-\frac
{1}{6s})N-(n-2)\rfloor))}{N}\\
& =\lim_{N\rightarrow\infty}\frac{2}{N}\sum_{k=1}^{\lfloor(\frac{p}{s}%
-\frac{1}{6s})N-(n-2)\rfloor}\log( 2 |\sin\frac{(n-2+k)s\pi}{N}|)\nonumber
\end{align}

Set $(\frac{p}{s}-\frac{1}{6s})N-(n-2)=M$, then
\begin{align}
\frac{1}{N}=\frac{1}{M+n-2}(\frac{p}{s}-\frac{1}{6s}),
\end{align}
and
\begin{align}
&  \frac{2}{N}\sum_{k=1}^{\lfloor(\frac{p}{s}-\frac{1}{6s})N-(n-2)\rfloor}%
\log(2|\sin\frac{(n-2+k)s\pi}{N}|)\\
&  =\frac{2}{M+n-2}(\frac{p}{s}-\frac{1}{6s})\sum_{k=1}^{\lfloor M\rfloor}%
\log\left(  2|\sin\frac{(n-2+k)(p-\frac{1}{6})\pi}{M+(n-2)}|\right)
\nonumber\\
&  =\frac{2}{M+n-2}(\frac{p}{s}-\frac{1}{6s})\sum_{l=n-1}^{\lfloor
M\rfloor+n-2}\log\left(  2|\sin\frac{l(p-\frac{1}{6})\pi}{M+(n-2)}|\right)
\nonumber\\
&  =\frac{2}{s\pi}\frac{1}{M+n-2}(p-\frac{1}{6})\pi\sum_{l=1}^{\lfloor
M\rfloor+n-2}\log\left(  2\sin|\frac{l(p-\frac{1}{6})\pi}{M+(n-2)}|\right)
\nonumber\\
&  -\frac{2}{M+n-2}(\frac{p}{s}-\frac{1}{6s})\sum_{l=1}^{n-2}\log\left(
2|\sin\frac{l(p-\frac{1}{6})\pi}{M+(n-2)}|\right)  \nonumber
\end{align}
For fixed $n,s$, when $N\rightarrow\infty$, the second term goes to zero, and
the first term is equal to the integral
\begin{align}
\frac{2}{s\pi}\int_{0}^{(p-\frac{1}{6})\pi}\log(2|\sin t|)dt.
\end{align}

Considering the Lobachevsky function
\begin{align}
\Lambda(x)=-\int_{0}^{x}\log(2|\sin t|)dt,
\end{align}
it has the period $\pi$, thus
\begin{align}
\Lambda((p-\frac{1}{6})\pi)=\Lambda(\frac{5\pi}{6}).
\end{align}
So we finish the proof for $\tilde{a}=0$ case.

For $\tilde{a}=n-2$ $(a=n-1)$, formula (3.15) gives
\begin{align}
f^{SU(n)}(N,k)=\left( 2\sin\frac{(n-2+k)s\pi}{N+n-2}\right) ^{2}.
\end{align}

Then we can finish the proof following the similar statement just as
$\tilde{a}=0$ case, or refer to \cite{Kaw1} for this case.
\end{proof}

\begin{proposition}
For the figure-eight knot $4_{1}$, part (i) of the Conjecture 3.3 holds; part
(ii) of the Conjecture 3.3 holds when $a=1$ or $a=n-1$.
\end{proposition}

\begin{proof}
The part (i) of Conjecture 3.3 is reduced to the congruent relation for
$4_{1}$ which has been proved in Theorem 7.11 in \cite{CLPZ}. So we only need
to prove part (ii) of Conjecture 3.3 when $a=1$ or $a=n-1$. By using the
property of the Lobachevsky function, we have
\begin{align}
\Lambda(\frac{5\pi}{6})=-\frac{3}{2}\Lambda(\frac{\pi}{3}).
\end{align}
Therefore, by Lemma 3.5
\begin{align}
2\pi s\lim_{N\rightarrow\infty}\frac{\log J_{N-1}^{SU(n)}(4_{1};\xi
_{N-1,a}(s))}{N} &  =4\int_{0}^{\frac{5\pi}{6}}\log2\sin(x)dx\\
&  =6\Lambda(\frac{\pi}{3})\nonumber\\
&  =Vol(S^{3}\setminus4_{1}).\nonumber
\end{align}
So Proposition 3.6 is proved.
\end{proof}

In order to prove part (ii) of the Conjecture 3.3 for figure-eight knot
$4_{1}$ completely, we first consider the case of $s=1$.

By the method used in the proof of the Lemma 3.5, the first step is to find a
maximal $k_{m}$, such that the function of $k$
\begin{align}
f^{SU(n)}(N,k)=\frac{\sin\frac{(n-2+k)\pi}{N+\tilde{a}}}{\sin\frac{k\pi
}{N+\tilde{a}}}4\left( \sin\frac{(k+\tilde{a})\pi}{N+\tilde{a}}\right) \left(
\sin\frac{((n-2)-\tilde{a}+k)\pi}{N+\tilde{a}}\right) ,
\end{align}
satisfies
\begin{align}
f^{SU(n)}(N,k_{m})\geq1, \ f^{SU(n)}(N,k_{m}+1)<1.
\end{align}

\begin{lemma}
Such $k_{m}$ must be in
\begin{align}
\lfloor\frac{5}{6}(N+\tilde{a})-2(n-2)\rfloor\leq k_{m}\leq\lfloor\frac{5}%
{6}(N+\tilde{a}) \rfloor.
\end{align}

\end{lemma}

\begin{proof}
The upper bound of $k_{m}$ is clear, in fact, if $k_{m}\geq\lfloor\frac{5}%
{6}(N+\tilde{a})\rfloor$, then $f^{SU(n)}(N,k_{m})<1$. Now we need to estimate
a lower bound of $k_{m}$. We can assume
\begin{align}
\frac{1}{2}\leq\frac{k_{m}}{N+\tilde{a}}\leq\frac{5}{6}.
\end{align}
Since
\begin{align}
\frac{\sin\frac{(n-2+k)\pi}{N+\tilde{a}}}{\sin\frac{k\pi}{N+\tilde{a}}}%
=\sin\frac{(n-2)\pi}{N+\tilde{a}}\cot\frac{k\pi}{N+\tilde{a}}+\cos
\frac{(n-2)\pi}{N+\tilde{a}},
\end{align}
Set $\frac{(n-2)\pi}{N+\tilde{a}}=\alpha$, for $\frac{1}{2}\leq\frac{k}%
{N+a}\leq\frac{5}{6}$, we have
\begin{align}
\frac{\sin\frac{(n-2+k)\pi}{N+\tilde{a}}}{\sin\frac{k\pi}{N+\tilde{a}}}%
\geq1-\frac{1}{2}\alpha^{2}-\sqrt{3}\alpha,
\end{align}
where we used the inequalities:
\begin{align}
\sin\alpha<\alpha\ \text{and}\ \cos\alpha>1-\frac{1}{2}\alpha^{2},\ \text{for
small}\ \alpha>0.
\end{align}

Since $\frac{k}{N+\tilde{a}}\geq\frac{1}{2}$, we have
\begin{align}
4 \sin\frac{(k+\tilde{a})\pi}{N+\tilde{a}}\sin\frac{(n-2-\tilde{a}+k)\pi
}{N+\tilde{a}}\geq4(\sin\frac{(n-2+k)\pi}{N+\tilde{a}})^{2}=4(\sin(\frac{5\pi
}{6}-\beta))^{2},
\end{align}
where $\beta=\frac{(\frac{5}{6}(N+\tilde{a})-(n-2)-k)\pi}{N+\tilde{a}}$.

Therefore,
\begin{align}
4 \sin\frac{(k+\tilde{a})\pi}{N+\tilde{a}}\sin\frac{(n-2-\tilde{a}+k)\pi
}{N+\tilde{a}}\geq4(\sin(\frac{5\pi}{6}-\beta))^{2}=1+2\sqrt{3}\sin\beta
+2\sin^{2}\beta.
\end{align}

Combing (3.33) and (3.36), we have
\begin{align}
& \frac{\sin\frac{(n-2+k)\pi}{N+\tilde{a}}}{\sin\frac{k\pi}{N+\tilde{a}}%
}4\left( \sin\frac{(k+\tilde{a})\pi}{N+\tilde{a}}\right) \left( \sin
\frac{((n-2)-\tilde{a}+k)\pi}{N+\tilde{a}}\right) \\
& \geq(1-\sqrt{3}\alpha-\frac{1}{2}\alpha^{2})(1+2\sqrt{3}\sin\beta+2\sin
^{2}\beta)\nonumber\\
& =1+\sqrt{3}(2\beta-\alpha)+O(\alpha^{2})+O(\beta^{2})\nonumber
\end{align}
when $\alpha, \beta$ are both small enough.

If we let $k_{0}=\frac{5}{6}(N+\tilde{a})-2(n-2)$, then $\beta=\alpha$. By
(3.37), we have $f^{SU(n)}(N,k_{0})>1$. Hence, by the definition of $k_{m}$,
we must have $k_{m}\geq k_{0}$.
\end{proof}

Now, let us consider the case of general $s$. Without loss of generality, we
only need to prove the case that $N,s$ are coprime. For $k=1$ to $N-1$, by
formula $(3.15)$, The function $f^{SU(n)}(N,k)$ of $k$ has certain period.
Thus, as a function of $j$, one can show that $g^{SU(n)}(N,j)$ may take the
maximal value at $j_{1}=k_{m}^{(1)}$ or $j_{2}=k_{m}^{(2)}$,...,$j_{p}%
=k_{m}^{(p)}$,... or $j_{s}=k_{m}^{(s)}$, where
\begin{align}
\lfloor\frac{6p-1}{6s}(N+\tilde{a})-2(n-2)\rfloor\leq k_{m}^{(p)}\leq
\lfloor\frac{6p-1}{6s}(N+\tilde{a})\rfloor.
\end{align}

Without loss of generality, we may assume $g^{SU(n)}(N,j)$ take the maximal
value at $j_{p_{0}}=k_{m}^{(p_{0})}$, where $1\leq p_{0}\leq s$. Then
$g^{SU(n)}(N,k_{m}^{(p_{0})})\geq g^{SU(n)}(N,k_{m}^{(1)})>0$ (In fact,
$g^{SU(n)}(N,k_{m}^{(1)})>1$ for large $N$, because the integral $\int
_{0}^{\frac{5\pi}{6}}\log(2\sin(t))dt>0$). Next, we will show the following

\begin{lemma}%
\begin{align}
g^{SU(n)}(N,k_{m}^{p_{0}})\leq J_{N-1}^{SU(n)}(4_{1})\leq N\cdot
g^{SU(n)}(N,k_{m}^{p_{0}}).
\end{align}

\end{lemma}

\begin{proof}
By definition
\begin{align}
J_{N-1}^{SU(n)}(4_{1})=1+\sum_{j=1}^{N-1}g^{SU(n)}(N,j)\leq N\cdot
g^{SU(n)}(N,k_{m}^{p_{0}})
\end{align}
on the other hand side, we let
\begin{align}
k^{(1)} &  =\lfloor\frac{(N+\tilde{a})}{s}\rfloor-(n-2),\\
k^{(2)} &  =\lfloor\frac{2(N+\tilde{a})}{s}\rfloor-(n-2),\nonumber\\
&  ....\nonumber\\
k^{(p_{0}-1)} &  =\lfloor(p_{0}-1)(\frac{N+\tilde{a}}{s})\rfloor
-(n-2).\nonumber
\end{align}
In the summation of $\sum_{j=1}^{N-1}g^{SU(n)}(N,j)$, the possible negative
terms are only in the following list:
\begin{align}
g^{SU(n)}(N,k^{(i)}+l),\ \text{for}\ i=1,..,p_{0}-1,\ \text{and}\ l=0,...,n-2.
\end{align}
Since $k^{(i)}+l=\lfloor\frac{i(N+\tilde{a})}{s}\rfloor-(n-2)+l=\frac
{i(N+\tilde{a})}{s}-\langle\frac{i(N+\tilde{a})}{s}\rangle-(n-2)+l$, where we
use $\langle\frac{i(N+\tilde{a})}{s}\rangle$ to denote the fractional part of
$\frac{i(N+\tilde{a})}{s}$, hence $0\leq\langle\frac{i(N+\tilde{a})}{s}%
\rangle<1$. Therefore
\begin{align}
\sin\frac{(x+k^{(i)}+l)s\pi}{N+\tilde{a}} &  =\sin(i\pi-\frac{(\langle
\frac{i(N+\tilde{a})}{s}\rangle+(n-2)-l-x)s\pi}{N+\tilde{a}})\\
&  =(-1)^{i-1}\sin(\frac{(\langle\frac{i(N+\tilde{a})}{s}\rangle
+(n-2)-l-x)s\pi}{N+\tilde{a}}),\nonumber
\end{align}
where $x$ stands for $0,n-2,n-2-\tilde{a}$ and $\tilde{a}$. So for large $N$,
$\sin\frac{(x+k^{(i)}+l)s\pi}{N+\tilde{a}}=O(\frac{1}{N+\tilde{a}})$. Thus, we
have
\begin{align}
f^{SU(n)}(N,k^{(i)}+l)=O((\frac{1}{N+\tilde{a}})^{2}),
\end{align}
then
\begin{align}
\overset{l}{\underset{j=0}{\prod}}f^{SU(n)}(N,k^{(i)}+j)=O((\frac{1}%
{N+a})^{2(l+1)})
\end{align}

and%
\begin{align}
1+\overset{n-2}{\underset{l=0}{\sum}}\overset{l}{\underset{j=0}{\prod}%
}f^{SU(n)}(N,k^{(i)}+j)>0.
\end{align}

Therefore,%

\begin{align}
g^{SU(n)}(N,k^{(i)}-1)+\underset{l=0}{\overset{n-2}{\sum}}g^{SU(n)}%
(N,k^{(i)}+l)=g^{SU(n)}(N,k^{(i)}-1)(1+\overset{n-2}{\underset{l=0}{\sum}%
}\overset{l}{\underset{j=0}{\prod}}f^{SU(n)}(N,k^{(i)}+j))>0
\end{align}

Thus we have%

\begin{align*}
J_{N}^{SU(n)}(4_{1})  & =\underset{%
\begin{array}
[c]{c}%
u=0\text{ and }u\neq k^{(i)}+l,\\
\text{ where }l=-1,0,...,n-2
\end{array}
}{\sum}g^{SU(n)}(N,u)\\
& +\overset{s}{\underset{i=0}{\sum}}\left(  g^{SU(n)}(N,k^{(i)}-1)+\overset
{s}{\underset{l=0}{\sum}}g^{SU(n)}(N,k^{(i)}+l)\right)  \\
& >\overset{N}{\underset{%
\begin{array}
[c]{c}%
u=0\text{ and }u\neq k^{(i)}+l,\\
\text{ where }l=-1,0,...,n-2
\end{array}
}{\sum}}g^{SU(n)}(N,u)
\end{align*}

It is easy to know that each term $g^{SU(n)}(N,u)$ in the above expression is
positive. For large $N$, it is impossible that $k^{(i)}+l=k_{m}^{(j)}$ for any
pair $(i,j,l)$, where $1\leq i,j\leq s,l=-1,0,...,n-2$. Thus we have
\begin{align}
J_{N}^{SU(n)}(4_{1})=\overset{N}{\underset{u=0}{\sum}}g^{SU(n)}(N,u)\geq
g^{SU(n)}(N,k_{m}^{(p_{0})})
\end{align}

\end{proof}

\begin{proposition}
Part (ii) of the Conjecture 3.3 holds for the figure-eight knot $4_{1}$.
\end{proposition}

\begin{proof}
Now, we can finish the proof of proposition as follow:
\begin{align}
&  \lim_{N\rightarrow\infty}\frac{\log(J_{N-1}^{SU(n)}(4_{1}))}{N}\\
&  =\lim_{N\rightarrow\infty}\frac{\log(g^{SU(n)}(N,k_{m}^{(p_{0})}))}%
{N}\nonumber\\
&  =\lim_{N\rightarrow\infty}\frac{1}{N}\sum_{k=1}^{k_{m}^{(p_{0})}}%
\log\left(  4\frac{|\sin\frac{(n-2+k)s\pi}{N+\tilde{a}}|}{|\sin\frac{ks\pi
}{N+\tilde{a}}|}|\sin\frac{(k+\tilde{a})s\pi}{N+\tilde{a}}||\sin
\frac{((n-2)-\tilde{a}+k)s\pi}{N+\tilde{a}}|\right)  \nonumber
\end{align}
where we have used the fact $g^{SU(n)}(N,k_{m}^{(p_{0})})>0$, thus the number
of the negative term of the form $\sin\frac{(x+k)s\pi}{N+\tilde{a}}$ in the
product $g^{SU(n)}(N,k_{m}^{(p_{0})})$ must be even. For each term of above
with the form $\log\left(  |\sin\frac{(x+k)s\pi}{N+\tilde{a}}|\right)  $, by
the method in the proof of Lemma 3.5, we have
\begin{align}
\lim_{N\rightarrow\infty}\frac{1}{N}\sum_{k=1}^{k_{m}^{(p_{0})}}\log\left(
|\sin\frac{(x+k)s\pi}{N+\tilde{a}}|\right)  =\frac{1}{\pi}\int_{0}^{\frac
{5}{6}\pi}\log|\sin(t)|dt.
\end{align}
Finally, we obtain
\begin{align}
\lim_{N\rightarrow\infty}\frac{\log(J_{N-1}^{SU(n)}(4_{1}))}{N}=(\frac{p_{0}%
}{s}-\frac{1}{6s})\log4+\frac{2}{s\pi}\int_{0}^{(p_{0}-\frac{1}{6})\pi}%
\log|\sin(t)|dt=\frac{2}{s\pi}\int_{0}^{(p_{0}-\frac{1}{6})\pi}\log
(2|\sin(t)|)dt.
\end{align}
Finally we have
\begin{align}
2\pi s\lim_{N\rightarrow\infty}\frac{\log J_{N-1}^{SU(n)}(4_{1};\xi
_{N-1,a}(s))}{N}  & =4\int_{0}^{(p_{0}-\frac{1}{6})\pi}\log(2|\sin
(t)|)dt=-4\Lambda((p_{0}-\frac{1}{6})\pi)\\
& =-4\Lambda(\frac{5}{6}\pi)=6\Lambda(\frac{\pi}{3})=Vol(S^{3}\setminus4_{1}).
\end{align}

\end{proof}

\clearpage

\section{Appendix}
\subsection{Example of Conjectures 1.4}
We define
\begin{align}
Q(N,n,a,s)=2\pi s[\log J_{N}^{SU(n)}%
(\mathcal{K};\xi_{N,a}(s))-\log
J_{N-1}^{SU(n)}(\mathcal{K};\xi_{N-1,a}(s))],
\end{align}
and compute its limit $\lim_{N\rightarrow \infty} Q(N,n,a,s)$.
\subsubsection{The knot $5_2$}
We know that $Vol(S^{3}/5_{2})\approx2.828122$ and
$0.1532041333\ast2\pi^{2}=3.02413$. So the complex volume for
hyperbolic knot $5_{2}$ is $2.828122+3.02413\sqrt {-1}$. By using
the formula for $SU(n)$ invariant of $5_2$ in \cite{Kaw2}. We have
the following table. Here we set $s=1$.

$%
\begin{array}
[c]{cccc}%
N\backslash(n,a) & (2,1) & (3,1) & (3,2)\\
10 & 3.73795+2.62595\sqrt{-1} & 4.95561+1.83803\sqrt{-1} &
4.77077+2.02852\sqrt{-1}\\
20 & 3.27786+2.92530\sqrt{-1} & 3.90996+2.71482\sqrt{-1} &
3.85936+2.74299\sqrt
{-1}\\
30 & 3.13249+2.97960\sqrt{-1} & 3.55378+2.88496\sqrt{-1} &
3.53064+2.89368\sqrt
{-1}\\
40 & 3.05822+2.99885\sqrt{-1} & 3.37391+2.94535\sqrt{-1} &
3.36071+2.94911\sqrt{-1}\\
50 & 3.01308+3.00786\sqrt{-1} & 3.26546+2.97352\sqrt{-1} &
3.25694+2.97547\sqrt{-1}\\
70 & 2.96096+3.01577\sqrt{-1} & 3.14105+2.99819\sqrt{-1} &
3.13666+2.99891\sqrt{-1}\\
100 & 2.92148+3.02001\sqrt{-1} & 3.04745+3.01138\sqrt{-1} &
3.04528+3.01163\sqrt{-1}\\
200 & 2.87502+3.02309\sqrt{-1} & 2.93794+3.02093\sqrt{-1} &
2.93739+3.02096\sqrt{-1}%
\end{array}
$

\bigskip

$%
\begin{array}
[c]{cccc}%
N\backslash(n,a) & (4,1) & (4,2) & (4,3)\\
10 & 6.23105+0.579569\sqrt{-1} & 5.80661+0.952116\sqrt{-1} &
5.70074+1.31854\sqrt{-1}\\
20 & 4.56915+2.40633\sqrt{-1} & 4.43367+2.42953\sqrt{-1} &
4.41159+2.51307\sqrt{-1}\\
30 & 3.98932+2.74861\sqrt{-1} & 3.92527+2.74711\sqrt{-1} &
3.91657+2.78186\sqrt{-1}\\
40 & 3.69821+2.86884\sqrt{-1} & 3.66116+2.86452\sqrt{-1} &
3.65662+2.88323\sqrt{-1}\\
50 & 3.52356+2.92461\sqrt{-1} & 3.49946+2.92049\sqrt{-1} &
3.49671+2.93210\sqrt
{-1}\\
70 & 3.32419+2.97327\sqrt{-1} & 3.31168+2.97036\sqrt{-1} &
3.31036+2.97605\sqrt{-1}\\
100 & 3.17494+2.99918\sqrt{-1} & 3.16873+2.99744\sqrt{-1} &
3.16812+3.00014\sqrt{-1}\\
200 & 3.00124+3.01788\sqrt{-1} & 2.99967+3.01736\sqrt{-1} &
2.99953+3.01800\sqrt{-1}%
\end{array}
$

From the above tables,  one can see that the $SU(2)$ invariants, i.e., the colored Jones polynomials, converge to the complex volume faster
than the general $SU(n)$ invariants.   As to the $SU(n)$ invariants,
one can see that the $SU(n)$ invariants converge to the complex
volume at $a=n-1$ faster than at the other values  $a=1,..,n-2$.

\subsection{Examples of Conjecture 1.3}

\smallskip

Case of knot $5_{2}$

\smallskip

The $SU(2)$ invariant:

$H_{0}=1$

$H_{1}=-q^{4}(1+q^{4})$

$H_{2}=q^{10}(1+q^{4}+q^{6}+q^{12})$

$H_{3}=-q^{18}(1+q^{4}+q^{6}+q^{8}+q^{12}+q^{14}+q^{16}+q^{24})$

$H_{4}=q^{28}(1+q^{4}+q^{6}+q^{8}+q^{10}+q^{12}+q^{14}+2q^{16}+q^{18}+q^{20}%
+q^{24}+q^{26}+q^{28}+q^{30}+q^{40})$

\bigskip

The $SU(3)$ invariant:

$H_{0}=1$

$H_{1}=-q^{5}(1+q^{2})^{2}(1-q^{2}+q^{4})$

$H_{2}=q^{12}(1+q^{2}+q^{4})(1+q^{6}+q^{8}+q^{16})$

$H_{3}=-q^{21}(1+q^{2})^{2}(1+q^{4})(1-q^{2}+q^{4}+q^{8}+q^{16}+q^{20}%
-q^{22}+q^{24}-q^{26}+q^{28})$

$H_{4}=q^{32}(1+q^{2}+q^{4}+q^{6}+q^{8})(1+q^{6}+q^{8}+q^{10}+q^{12}%
+q^{16}+q^{18}+2q^{20}+q^{22}+q^{24}+q^{30}+q^{32}+q^{34}+q^{36}+q^{48})$

\bigskip

The $SU(4)$ invariant:

$H_{0}=1$

$H_{1}=-q^{6}(1+q^{2}+q^{4}+q^{8}+q^{10}+q^{12})$

$H_{2}=q^{14}(1+q^{2}+q^{4})(1+q^{4}+q^{8}+q^{10}+q^{12}+q^{14}+q^{20}%
+q^{24})$

$H_{3}=-q^{24}(1+q^{4})^{2}(1+q^{2}+2q^{8}+2q^{10}+q^{12}+q^{14}%
+2q^{16}+q^{18}+q^{22}+3q^{24}+2q^{26}+q^{32}+q^{38}+q^{40})$

$H_{4}=q^{36}(1+q^{2}+q^{4})(1+q^{2}+q^{4}+q^{6}+q^{8})(1-q^{2}+q^{4}%
+q^{8}+q^{12}+q^{14}+q^{18}+q^{20}+2q^{24}+2q^{28}+q^{32}+q^{36}+q^{40}%
+q^{42}+q^{46}+q^{56}-q^{58}+q^{60})$

\bigskip

Case of knot $6_{1}$

\smallskip

The $SU(2)$ invariant:

$H_{0}=1$

$H_{1}=1+q^{4}$

$H_{2}=1+q^{4}+q^{6}+q^{12}$

$H_{3}=1+q^{4}+q^{6}+q^{8}+q^{12}+q^{14}+q^{16}+q^{24}$

$H_{4}=1+q^{4}+q^{6}+q^{8}+q^{10}+q^{12}+q^{14}+2q^{16}+q^{18}+q^{20}%
+q^{24}+q^{26}+q^{28}+q^{30}+q^{40}$

\bigskip

The $SU(3)$ invariant:

$H_{0}=1$

$H_{1}=q^{-1}(1+q^{2})^{2}(1-q^{2}+q^{4})$

$H_{2}=q^{-2}(1+q^{2}+q^{4})(1+q^{6}+q^{8}+q^{16})$

$H_{3}=q^{-3}(1+q^{2})^{2}(1+q^{4})(1-q^{2}+q^{4}+q^{8}+q^{16}+q^{20}%
-q^{22}+q^{24}-q^{26}+q^{28})$

$H_{4}=q^{-4}(1+q^{2}+q^{4}+q^{6}+q^{8})(1+q^{6}+q^{8}+q^{10}+q^{12}%
+q^{16}+q^{18}+2q^{20}+q^{22}+q^{24}+q^{30}+q^{32}+q^{34}+q^{36}+q^{48})$

\bigskip

The $SU(4)$ invariant:

$H_{0}=1$

$H_{1}=q^{-2}(1+q^{2}+q^{4}+q^{8}+q^{10}+q^{12})$

$H_{2}=q^{-4}(1+q^{2}+q^{4})(1+q^{4}+q^{8}+q^{10}+q^{12}+q^{14}+q^{20}%
+q^{24})$

$H_{3}=q^{-6}(1+q^{4})^{2}(1+q^{2}+2q^{8}+2q^{10}+q^{12}+q^{14}+2q^{16}%
+q^{18}+q^{22}+3q^{24}+2q^{26}+q^{32}+q^{38}+q^{40})$

$H_{4}=q^{-8}(1+q^{2}+q^{4})(1+q^{2}+q^{4}+q^{6}+q^{8})(1-q^{2}+q^{4}%
+q^{8}+q^{12}+q^{14}+q^{18}+q^{20}+2q^{24}+2q^{28}+q^{32}+q^{36}+q^{40}%
+q^{42}+q^{46}+q^{56}-q^{58}+q^{60})$

\end{document}